\documentclass{amsart}
\usepackage{amssymb, amsmath,amsthm }
\usepackage{stmaryrd}
\usepackage[all]{xy}
\newtheorem{prop}{Proposition}[section]

\newtheorem{lem}[prop]{Lemma}
\newtheorem{thm}[prop]{Theorem}

\newtheorem{remar}[prop]{Remark}
\newtheorem{cor}[prop]{Corollary}

\DeclareMathAlphabet{\mathpzc}{OT1}{pzc}{m}{it}
\DeclareMathOperator{\Aut}{Aut}
\DeclareMathOperator{\End}{End}
\DeclareMathOperator{\Hom}{Hom}

\DeclareMathOperator{\Ind}{Ind}
\DeclareMathOperator{\cInd}{c-Ind}

\DeclareMathOperator{\Sym}{Sym}

\DeclareMathOperator{\GL}{GL}

\DeclareMathOperator{\soc}{soc}

\DeclareMathOperator{\Irr}{Irr}

\DeclareMathOperator{\id}{id}
\DeclareMathOperator{\Mod}{Mod}

\DeclareMathOperator{\val}{val}
\DeclareMathOperator{\Sp}{Sp}

\DeclareMathOperator{\Ext}{Ext}

\DeclareMathOperator{\dt}{det}
\DeclareMathOperator{\Ban}{Ban}

\newcommand{\cIndu}[3]{\cInd_{#1}^{#2}{#3}}

\newcommand{\Indu}[3]{\Ind_{#1}^{#2}{#3}}

\newcommand{\Qp}{\mathbb {Q}_p}
\newcommand{\Zp}{\mathbb{Z}_p}

\newcommand{\GG}{\mathcal G}
\newcommand{\HH}{\mathcal H}
\newcommand{\NN}{\mathbb N}
\newcommand{\Eins}{\mathbf 1}

\newcommand{\KK}{\mathfrak K}
\newcommand{\ZZ}{\mathbb Z}

\newcommand{\Fp}{\mathbb F_p}

\newcommand{\wP}{\widetilde{P}}

\newcommand{\OO}{\mathcal O}

\newcommand{\BB}{\mathfrak B}

\newcommand{\br}[1]{\llbracket #1\rrbracket}

\newcommand{\sm}{\mathrm{sm}}

\newcommand{\adm}{\mathrm{adm}}
\newcommand{\alg}{\mathrm{alg}}

\title{Blocks for mod $p$ representations of $\GL_2(\Qp)$}
\author{Vytautas Pa\v{s}k\={u}nas}
\date{\today.}
\begin{document} 

\begin{abstract} Let $\pi_1$ and $\pi_2$ be absolutely irreducible smooth representations of $G=\GL_2(\Qp)$ with a central character,
defined over a finite extension of $\mathbb F_p$. We show that if there exists a non-split extension between $\pi_1$ and $\pi_2$ 
then they both appear as subquotients of the reduction modulo $p$ of a unit ball in a crystalline Banach space representation 
of $G$. The results of Berger-Breuil describe such reductions and 
allow us to organize the irreducible representation into blocks. The result is new for $p=2$, the proof, which works for all $p$, 
 is new.  
\end{abstract}
\maketitle
\section{Introduction}
Let $L$ be a finite extension of $\Qp$, with the ring of integers $\OO$, a uniformizer $\varpi$, and residue field $k$, and let 
$G=\GL_2(\Qp)$ and let $B$ be the subgroup of upper-triangular matrices in $G$.

\begin{thm}\label{main} Let $\pi_1$, $\pi_2$ be smooth, absolutely irreducible $k$-representations of $G$ with a central character. 
Suppose that $\Ext^1_G(\pi_2,\pi_1)\neq 0$ then after replacing $L$ by a finite extension, we may find 
integers $(l, k)\in \ZZ\times \NN$ and  unramified characters $\chi_1, \chi_2:\Qp^{\times}\rightarrow L^{\times}$ with 
$\chi_2\neq \chi_1|\centerdot|$, such that $\pi_1$ and $\pi_2$ are subquotients of $\overline{\Pi}^{ss}$, where $\overline{\Pi}^{ss}$
is the semi-simplification of the reduction modulo $\varpi$ of an open bounded $G$-invariant lattice in $\Pi$, where $\Pi$
is the universal unitary completion of 
$$(\Indu{B}{G}{\chi_1\otimes\chi_2|\centerdot|^{-1}})_{\sm}\otimes \dt^l \otimes \Sym^{k-1} L^2.$$
\end{thm}

The results of Berger-Breuil \cite{bb}, Berger \cite{berger}, Breuil-Emerton \cite{be} and 
\cite{except} describe explicitly the possibilities for $\overline{\Pi}^{ss}$, see Proposition \ref{universal}. These results  
and the Theorem imply that $\Ext^1_G(\pi_2,\pi_1)$ vanishes in many cases. Let us make this more precise. 

Let $\Mod^{\sm}_G(\OO)$ be the category of smooth $G$-representation on $\OO$-torsion modules. It contains
$\Mod^{\sm}_G(k)$, the category of smooth $G$-representations on $k$-vector spaces, as a full subcategory. 
Every irreducible object $\pi$ of 
$\Mod^{\sm}_G(\OO)$ is killed by $\varpi$, and hence is an object of $\Mod^{\sm}_G(k)$. 
Barthel-Livn\'e \cite{bl} and Breuil \cite{breuil1} have classified the absolutely 
irreducible smooth representations $\pi$ admitting a central character. They fall into four disjoint classes:
\begin{itemize}
\item[(i)] characters $\delta\circ \det$;
\item[(ii)] special series $\Sp\otimes\delta\circ \det$;
\item[(iii)] principal series $(\Indu{B}{G}{\delta_1\otimes\delta_2})_{\sm}$, $\delta_1\neq \delta_2$;
\item[(iv)] supersingular representations,
\end{itemize}
where $\Sp$ is the Steinberg representation, that is 
the locally constant functions from $\mathbb P^1(\Qp)$ to $k$ modulo the constant functions; 
$\delta, \delta_1, \delta_2: \Qp^{\times}\rightarrow k^{\times}$ are smooth characters 
and we consider $\delta_1\otimes\delta_2$
as a character of $B$, which sends $\bigl (\begin{smallmatrix} a & b \\ 0 & d \end{smallmatrix} \bigr )$ to $\delta_1(a)\delta_2(d)$.
Using their results and some easy arguments, see \cite[\S5.3]{cmf},
 one may show that for an irreducible smooth representations $\pi$ the following are equivalent: 1) $\pi$ is admissible, 
which means that $\pi^H$ is finite dimensional for all open subgroups $H$ of $G$; 2) $\End_G(\pi)$ is finite dimensional 
over $k$; 3) there exists a finite extension $k'$ of $k$, such that $\pi\otimes_k k'$ is isomorphic to 
a finite direct sum of distinct absolutely irreducible $k'$-representations with a central character.

Let $\Mod^{\mathrm{l.adm}}_G(\OO)$  be the full subcategory of $\Mod^{\sm}_G(\OO)$, 
consisting of representations, 
which are equal to the union of their admissible subrepresentations. The categories $\Mod^{\sm}_G(\OO)$ and $\Mod^{\mathrm{l.adm}}_G(\OO)$ 
are abelian, see \cite[Prop.2.2.18]{ord1}. We define $\Mod^{\mathrm{l.adm}}_G(k)$ in exactly the same way with $\OO$ replaced by $k$. 
Let $\Irr_G^{\adm}$ be the set of irreducible  representation in $\Mod^{\mathrm{l.adm}}_G(\OO)$, then 
$\Irr_G^{\adm}$  is the set of irreducible representations in 
$\Mod^{\sm}_G(\OO)$ satisfying the equivalent conditions described above.  
We define an equivalence relation $\sim$ on $\Irr_G^{\adm}$: $\pi\sim \tau$, if there exists a sequence of irreducible 
admissible representations 
$\pi=\pi_1, \pi_2, \ldots, \pi_n=\tau$, such that for each $i$ one of the following holds: 1) $\pi_i\cong \pi_{i+1}$;
2) $\Ext^1_G(\pi_i, \pi_{i+1})\neq 0$; 3) $\Ext^1_G(\pi_{i+1}, \pi_{i})\neq 0$. We note that it does not matter for the 
definition of $\sim$, whether we compute $\Ext^1_G$ in $\Mod^{\sm}_G(\OO)$, $\Mod^{\sm}_G(k)$, $\Mod^{\mathrm{l.adm}}_G(\OO)$ 
or  $\Mod^{\mathrm{l.adm}}_G(k)$, since we only care about vanishing or non-vanishing of $\Ext^1_G(\pi_i, \pi_{i+1})$ for distinct 
irreducible representations. A block is an equivalence class of $\sim$.

\begin{cor}\label{blocks} The blocks containing an absolutely irreducible representation are given by the following:

\begin{itemize} 
\item[(i)] $\BB=\{\pi\}$ with $\pi$ supersingular;
\item[(ii)] $\BB=\{(\Indu{B}{G}{\delta_1\otimes\delta_2\omega^{-1}})_{\sm},
(\Indu{B}{G}{\delta_2\otimes\delta_1\omega^{-1}})_{\sm}\}$ with $\delta_2 \delta_1^{-1}\neq \omega^{\pm 1}, \Eins$;
\item[(iii)] $p>2$ and $\BB=\{ (\Indu{B}{G}{\delta\otimes\delta\omega^{-1}})_{\sm}\}$;
\item[(iv)] $p=2$ and $\BB=\{\Eins,  \Sp\}\otimes\delta\circ \det$;
\item[(v)] $p\ge 5$ and $\BB=\{\Eins, \Sp, (\Indu{B}{G}{\omega\otimes \omega^{-1}})_{\sm}\}\otimes \delta\circ \det$;
\item[(vi)] $p=3$ and $\BB=\{\Eins, \Sp, \omega\circ \det, \Sp\otimes \omega\circ \det\}\otimes \delta\circ \det$;
\end{itemize}
where $\delta, \delta_1, \delta_2: \Qp^{\times}\rightarrow k^{\times}$ are smooth characters and where $\omega: 
\Qp^{\times}\rightarrow k^{\times}$ is the character $\omega(x)= x|x| \pmod{\varpi}$.
\end{cor}
One may view the cases (iii) to (vi) as degenerations of case (ii). A finitely generated smooth admissible 
representation of $G$ is of finite length, \cite[Thm.2.3.8]{ord1}. This makes $\Mod^{\mathrm{l.adm}}_G(\OO)$ into a locally finite category. It follows from 
\cite{gabriel} that every locally finite category decomposes into blocks. In our situation we obtain:
\begin{equation}\label{blocksdecompose}
 \Mod^{\mathrm{l.adm}}_G(\OO)\cong \prod_{\BB\in \Irr_G^{\adm}/\sim} \Mod^{\mathrm{l.adm}}_{G}(\OO)[\BB],
\end{equation}
where $\Mod^{\mathrm{l.adm}}_{G}(\OO)[\BB]$ is the full subcategory of  $\Mod^{\mathrm{l.adm}}_G(\OO)$ consisting of representations, with 
all irreducible subquotients in $\BB$. One can deduce a similar result for the category of admissible unitary $L$-Banach space 
representations of $G$, see \cite[Prop.5.32]{cmf}. 

The result has been previously known for $p>2$. Breuil and the author \cite[\S8]{bp}, Colmez \cite[\S VII]{colmez}, 
Emerton \cite[\S 4]{ord2} and the author \cite{ext} have 
computed $\Ext^1_G(\pi_2, \pi_1)$ by different characteristic $p$ methods, which 
do not work in the exceptional cases, when $p=2$. In this paper, we go via characteristic $0$ and make use of a deep Theorem 
of Berger-Breuil. The proof is less involved, but it does not give any information about the extensions between irreducible 
representations lying in the same block. 

The motivation for these calculations comes from the $p$-adic Langlands correspondence for $\GL_2(\Qp)$. 
Colmez  in \cite{colmez} to a $2$-dimensional absolutely irreducible $L$-representation 
of the absolute Galois group of $\Qp$ has associated an admissible unitary absolutely irreducible 
non-ordinary $L$-Banach space representation of $G$. He showed that his construction induces 
an injection on the  isomorphism classes and asked whether it is a bijection, see \cite[\S 0.13]{colmez}.  
This has been answered affirmatively  in \cite{cmf} for $p\ge 5$, where the knowledge 
of blocks has been used in an essential way. The results of this paper should be useful in dealing with the remaining cases. 

Let us give a rough sketch of the argument. Let $0\rightarrow \pi_1\rightarrow \pi\rightarrow \pi_2\rightarrow 0$ be a non-split extension.
The method of \cite{bp} allows us to embed $\pi$ into $\Omega$, such that $\Omega|_K$ is admissible and an injective object in 
$\Mod^{\sm}_K(k)$, where $K=\GL_2(\Zp)$. Using the results of \cite{comp} we may lift $\Omega$ to an admissible unitary 
$L$-Banach space representation $E$ of $G$, in the sense that we may find a $G$-invariant unit ball $E^0$ in $E$, such that 
$E^0/\varpi E^0\cong \Omega$. Moreover, $E|_K$ is isomorphic to a direct summand of $\mathcal C(K, L)^{\oplus r}$, where 
$\mathcal C(K,L)$ is the space of continuous function with the supremum norm. This implies, using an argument of Emerton, 
that the $K$-algebraic vectors are dense in $E$. As a consequence we find a closed $G$-invariant subspace $\Pi$ of $E$,  
such that the reduction of $\Pi \cap E^0$ modulo $\varpi$ contains $\pi$ as a subrepresentation, and $\Pi$ contains 
$\oplus_{i=1}^m \frac{\cIndu{KZ}{G}{\tilde{\Eins}_i}}{(T-a_i)^{n_i}}\otimes \dt^{l_i} \otimes \Sym^{k_i-1} L^2$ as a dense subrepresentation, 
where $Z$ is the centre of $G$, $\tilde{\Eins}_i: KZ\rightarrow L^{\times}$ is a character, trivial on $K$, $a_i\in L$, and $T$ is a certain Hecke operator 
in $\End_G(\cIndu{KZ}{G}{\tilde{\Eins}_i})$, such that  $\frac{\cIndu{KZ}{G}{\tilde{\Eins}_i}}{(T-a_i)}$ is an unramified principal series 
representation. Once we have this we are in a good shape to prove Theorem \ref{main}.  

\textit{Acknowledgements.} I thank the anonymous referee for the comments, which led to an improvement of the exposition, and Jochen Heinloth
for a stimulating discussion. 

\section{Notation}
Let $L$ be a finite extension of $\Qp$ with the ring of integers $\OO$, uniformizer $\varpi$ and residue field $k$. 
We normalize the valuation $\val$ on $L$ so that $\val(p)=1$, and the norm $|\centerdot|$, so that $|x|=p^{-\val(x)}$, for all $x\in L$. 
Let $G=\GL_2(\Qp)$; $Z$ the centre of $G$; $B$ the subgroup of upper triangular matrices; $K=\GL_2(\Zp)$; 
$I=\{g\in K: g\equiv \bigl (\begin{smallmatrix} * & * \\ 0 & *\end{smallmatrix}\bigr )\pmod{p}\}$;
$I_1=\{g\in K:g\equiv \bigl (\begin{smallmatrix} 1 & * \\ 0 & 1\end{smallmatrix}\bigr )\pmod{p}\}$; 
 let $\KK$ be the $G$-normalizer of $I$; let 
$H=\{\bigl (\begin{smallmatrix} [\lambda] & 0 \\ 0 & [\mu] \end{smallmatrix}\bigr )$: $\lambda, \mu \in \Fp^{\times}\}$, 
where $[\lambda]$ is the Teichm\"uller lift of $\lambda$; 
let  $\mathcal G$ be the subgroup of $G$ generated by matrices
$\bigl (\begin{smallmatrix} p & 0 \\ 0 & p \end{smallmatrix}\bigr )$, 
$\bigl (\begin{smallmatrix} 0 & 1  \\  p & 0\end{smallmatrix}\bigr )$ and $H$. Let 
$G^+=\{g\in G: \val(\det(g))\equiv 0 \pmod{2}\}$. Since we are working with representations of locally pro-$p$ 
groups in characteristic $p$, these representations will not be  semi-simple in general; socle is the maximal 
semi-simple subobject. So for example, $\soc_G \tau$ means the maximal semi-simple $G$-subrepresentation of $\tau$.
Let $\Ban^{\adm}_G(L)$ be the category of admissible unitary $L$-Banach space representations of $G$, studied in \cite{iw}. This 
category is abelian. Let $\Pi$ be an admissible unitary $L$-Banach space representation of $G$, and let $\Theta$ be an open bounded
$G$-invariant lattice in $\Pi$, then $\Theta/\varpi \Theta$ is a smooth admissible
$k$-representation of $G$. If $\Theta/\varpi \Theta$ is of finite length as a $G$-representation, then 
we let $\overline{\Pi}^{ss}$ be the semi-simplification of $\Theta/\varpi \Theta$. Since any two such $\Theta$'s are commensurable, 
$\overline{\Pi}^{ss}$ is independent of the choice of $\Theta$. Universal unitary completions  are discussed in \cite[\S1]{pem}.

\section{Main}
Let $\pi_1$, $\pi_2$ be distinct smooth absolutely irreducible $k$-representation of $G$ with a central character.  
It follows from \cite{bl} and \cite{breuil1} that $\pi_1$ and $\pi_2$ are admissible. 
We suppose that there exists a non-split extension in $\Mod^{\sm}_G(\OO)$: 
\begin{equation}\label{exten}
0\rightarrow \pi_1\rightarrow \pi\rightarrow \pi_2\rightarrow 0.
\end{equation} 
Since $\pi_1$ and $\pi_2$ are distinct and irreducible, 
by examining the long exact sequence induced by multiplication with $\varpi$, we 
deduce that $\pi$ is killed by $\varpi$. A similar argument shows that   the existence of a non-split extension implies that 
the central character of $\pi_1$ is equal to the central character 
of $\pi_2$. Moreover, $\pi$ also has a central character, which is then equal to the central character of $\pi_1$, see 
\cite[Prop.8.1]{ext}. We denote this central character by $\zeta: Z\rightarrow k^{\times}$. After replacing $L$ by 
a quadratic extension and twisting 
by a character we may assume that $\zeta(\bigl( \begin{smallmatrix} p & 0\\ 0 & p\end{smallmatrix}\bigr))=1$.

\begin{lem}\label{invariants} If $\pi_1^{I_1}\neq \pi^{I_1}$ then Theorem \ref{main} holds for $\pi_1$ and $\pi_2$.
\end{lem}
\begin{proof} 
Since $\zeta$ is continuous, it is trivial on the pro-$p$ group $Z\cap I_1$. We thus may extend $\zeta$
to $ZI_1$, by letting $\zeta(zu)=\zeta(z)$ for all $z\in Z$, $u\in I_1$. If $\tau$ is a smooth $k$-representation 
of $G$ with a central character $\zeta$ then 
$\tau^{I_1}\cong \Hom_{I_1Z}(\zeta, \tau)\cong \Hom_G(\cIndu{KZ}{G}{\zeta}, \tau)$. Thus 
$\tau^{I_1}$ is naturally an $\HH:=\End_G(\cIndu{I_1 Z}{G}{\zeta})$ module. Taking 
$I_1$-invariants of \eqref{exten} we get an exact sequence of $\HH$-modules:
\begin{equation}\label{inv}
0\rightarrow \pi_1^{I_1}\rightarrow \pi^{I_1}\rightarrow \pi_2^{I_1}.
\end{equation} 
Since $\pi_2$ is irreducible, $\pi^{I_1}_2$ is an irreducible $\HH$-module by \cite{vigneras}. Hence, if $\pi_1^{I_1}\neq \pi^{I_1}$, then 
the last arrow is surjective. It is shown in \cite{ollivier}, that if $\tau$ is a smooth $k$-representation of $G$, 
with a central character $\zeta$, generated as a $G$-representation by its $I_1$-invariants, then the natural map 
$\tau^{I_1}\otimes_{\HH} \cIndu{KZ}{G}{\zeta}\rightarrow \tau$ is an isomorphism. This implies that the sequence 
$0\rightarrow \pi^{I_1}_1\rightarrow \pi^{I_1}\rightarrow \pi_2^{I_1}\rightarrow 0$ is non-split, and hence defines a non-zero element 
of $\Ext^1_{\HH}(\pi_2^{I_1}, \pi_1^{I_1})$.  
Since $\pi_i\cong \pi_i^{I_1}\otimes_{\HH} \cIndu{KZ}{G}{\zeta}$ for $i=1, 2$, the $\HH$-modules 
$\pi_1^{I_1}$ and $\pi_2^{I_1}$ are non-isomorphic. Non-vanishing of $\Ext^1_{\HH}(\pi_2^{I_1}, \pi_1^{I_1})$ implies that 
there exists a smooth character $\eta:G\rightarrow k^{\times}$ such that either ($\pi_1\cong \eta$ and $\pi_2\cong \Sp\otimes \eta$)
or  ($\pi_2\cong \eta$ and $\pi_1\cong \Sp\otimes \eta$), \cite[Lem.5.24]{cmf}, where $\Sp$ is the Steinberg representation. 
In both cases the universal unitary completion of $(\Indu{B}{G}{ |\centerdot|\otimes| \centerdot|^{-1}})_{\sm} \otimes\tilde{\eta}$, 
where $\tilde{\eta}: G\rightarrow \OO^{\times}$ is any smooth character lifting $\eta$, will satisfy the conditions of Theorem \ref{main} by \cite[5.3.18]{emcoates}.
\end{proof}
  
Lemma \ref{invariants} allows to assume that $\pi^{I_1}=\pi^{I_1}_1$. We note that this implies that $\soc_K \pi_1\cong \soc_K\pi$, and, 
since $I_1$ is contained in $G^+$, the restriction of \eqref{exten} to $G^+$ is a non-split extension of $G^+$-representations. 

Now we perform a renaming trick, the purpose of which is to get around some technical issues, when $p=2$. 
If either $p>2$ or $p=2$ and $\pi_1$ is neither a special series nor a character then we let $\tau_1=\pi_1$, $\tau=\pi$ and $\tau_2=\pi_2$.
If $p=2$ and $\pi_1$ is either a special series representation or a character, then we let 
$0\rightarrow \tau_1\rightarrow \tau\rightarrow\tau_2\rightarrow 0$ be the exact sequence obtained by tensoring \eqref{exten} with 
$\Indu{G^+}{G}{\Eins}$. 
In particular, $\tau\cong \pi\otimes\Indu{G^+}{G}{\Eins}$, which implies that $\tau|_{G^+}\cong \pi|_{G^+}\oplus \pi|_{G^+}$ and 
$\tau_1|_{G^+}\cong \pi_1|_{G^+}\oplus \pi_1|_{G^+}$. 
Hence, $\tau^{I_1}= \tau_1^{I_1}$ and $\soc_K \tau\cong \soc_K \tau_1\cong \soc_K \pi_1\oplus \soc_K \pi_1$. This implies that $\soc_G \tau\cong 
\soc_G \tau_1$. 

\begin{lem}\label{socles} $\soc_G \tau\cong \soc_G \tau_1\cong \pi_1$.
\end{lem}

\begin{proof} We already know that 
$\soc_G\tau\cong \soc_G \tau_1$ and we only need to consider the case $p=2$ and $\pi_1$ is either special series  or a character. 
The  assumption on $\pi_1$ implies that 
$\pi_1^{I_1}$ is one dimensional. Let $\mathfrak K$ be the normalizer of $I_1$ in $G$, then $I_1Z$ is a subgroup of $\mathfrak K$ of index $2$. 
We note that $I=I_1$  as $p=2$. Thus $\mathfrak K$ acts on $\pi_1^{I_1}$ by a character $\chi$, such that the restriction of 
$\chi$ to $I_1Z$ is equal to $\zeta$. Since $p=2$, we have an exact non-split sequence of $G$-representations 
$0\rightarrow \Eins \rightarrow \Indu{G^+}{G}{\Eins}\rightarrow \Eins \rightarrow 0$.  We note that $G^+$ and hence $ZI_1$ act trivially on all 
the terms in this sequence. By tensoring with $\pi_1$ we obtain an exact sequence
$0\rightarrow \pi_1\rightarrow \tau_1\rightarrow \pi_1\rightarrow 0$
of $G$-representations. Taking $I_1$-invariants, gives us an isomorphism of $\KK$-rep\-re\-sen\-ta\-tions 
$\tau_1^{I_1}\cong \pi^{I_1}_1\otimes\Indu{ZI_1}{\KK}{\Eins}$. This representation is a non-split extension of $\chi$ by itself. Thus 
$\tau_1$ is a non-split extension of $\pi_1$ by itself.  Hence, $\soc_G\tau_1\cong \pi_1$.   
\end{proof}
 
If $p=2$ then $\tau_1^{I_1}$ is $2$-dimensional and has a basis of the form $\{ v, \bigl (\begin{smallmatrix} 0 & 1\\ p & 0 \end{smallmatrix}
\bigr ) v \}$: if $\pi_1$ is either a character or special series, this follows from the isomorphism $\tau_1^{I_1}\cong \pi^{I_1}_1\otimes\Indu{ZI_1}{\KK}{\Eins}$, 
otherwise $\tau_1=\pi_1$ and the assertion follows from \cite[Cor. 6.4 (i)]{bp} noting that  the work of Bartel-Livn\'e \cite{bl} and 
Breuil \cite{breuil1} on classification of irreducible representations of $G$ implies that $\pi^{I_1}$ is isomorphic as a module of the pro-$p$ Iwahori Hecke algebra to $M(r, \lambda, \eta)$ defined in \cite[Def. 6.2]{bp}.
 Since $\tau^{I_1}=\tau_1^{I_1}$, \cite[Prop.9.2]{bp} implies that the inclusion $\tau^{I_1}\hookrightarrow \tau$ has 
a $\mathcal G$-equivariant section. 

\begin{prop}\label{omega} There exists a $G$-equivariant injection $\tau\hookrightarrow \Omega$, where $\Omega$ is a smooth $k$-representation 
of $G$, such that $\Omega|_K$ is an injective envelope of $\soc_K \tau$ in $\Mod^{\sm}_{K}(k)$, 
$\bigl (\begin{smallmatrix} p & 0 \\ 0 & p \end{smallmatrix}\bigr )$ acts trivially on $\Omega$  and $\Omega|_{\KK}\cong \Indu{\GG}{\KK}{\Omega^{I_1}}$.
\end{prop}
\begin{proof} The existence of $\Omega$ satisfying the first two conditions follows from  \cite[Cor.9.11]{bp}. The last condition is satisfied as a byproduct of the construction 
of the action of $\bigl (\begin{smallmatrix} 0 & 1 \\ p & 0 \end{smallmatrix}\bigr )$ in \cite[Lem. 9.6]{bp}.
\end{proof}

\begin{cor}\label{socleomega} Let $\Omega$ be as above then $\soc_K \Omega \cong \soc_K \tau_1$ and $\soc_G \Omega\cong \pi_1$.
\end{cor} 
\begin{proof} Since $\tau$ is a subrepresentation of $\Omega$, $\soc_K \tau$ is contained in $\soc_K \Omega$. Since $\Omega|_K$ is 
an injective envelope of $\soc_K \tau$, every non-zero $K$-invariant subspace of $\Omega$ intersects $\soc_K \tau$ non-trivially. 
This implies that $\soc_K\tau\cong \soc_K \Omega$. This implies the first assertion, as $\soc_K \tau\cong\soc_K \tau_1$. Moreover, every 
$G$-in\-va\-riant  non-zero subspace of $\Omega$ intersects $\tau$ non-trivially, since those are also $K$-invariant. This implies 
$\soc_G \Omega\cong \soc_G \tau\cong \pi_1$, where the last isomorphism follows from Lemma \ref{socles}.
\end{proof} 

\begin{lem}\label{lift_it} Let $\kappa$ be a finite dimensional $k$-representation of $\GG$ on which $\bigl (\begin{smallmatrix} p & 0 \\ 0 & p \end{smallmatrix}\bigr )$ acts trivially.
There exists an admissible unitary $L$-Banach space representation $(E, \|\centerdot\|)$ of $\KK$, such that 
$\| E\|\subset |L|$, $\bigl (\begin{smallmatrix} p & 0 \\ 0 & p \end{smallmatrix}\bigr )$
acts trivially on $E$, and the reduction modulo $\varpi$ of the unit ball in $E$ is isomorphic to $(\Indu{\GG}{\KK}{\kappa})_{\sm}$ as a $\KK$-representation.
\end{lem}
\begin{proof} It is enough to prove the statement, when $\kappa$ is indecomposable, which we now assume. Let $p^\ZZ$ be the subgroup of $G$ generated by $\bigl (\begin{smallmatrix} p & 0 \\ 0 & p \end{smallmatrix}\bigr )$. Since the order of $H$ is prime to $p$, and $H$ has index $2$ in $\GG/p^{\ZZ}$, 
$\kappa$ is either a character or an induction of a character from $H$ to $\GG/p^{\ZZ}$. In both cases we may lift $\kappa$ to a 
representation $\tilde{\kappa}^0$ of $\GG/p^{\ZZ}$ on a free $\OO$-module of rank $1$ or rank $2$ respectively. Let 
$\tilde{\kappa}=\tilde{\kappa}^0\otimes_{\OO} L$ and let $\|\centerdot\|$ be the gauge of  $\tilde{\kappa}^0$. Then $\|\centerdot\|$
is $\GG$-invariant and $\tilde{\kappa}^0$ is the unit ball with respect to $\|\centerdot\|$. Then 
$(\Indu{\GG/p^{\ZZ}}{\KK/p^{\ZZ}}{\tilde{\kappa}})_{\mathrm{cont}}$ with the norm $\|f\|_1:=\sup_{g\in \KK/p^{\ZZ}} \|f(g)\|$ is a
lift of $(\Indu{\GG/p^{\ZZ}}{\KK/p^{\ZZ}}{\kappa})_{\sm}$, where the subscript $\mathrm{cont}$ indicates continuous induction: 
the space of continuous functions with the right transformation property.
\end{proof}

\begin{thm}\label{lift} Let $\Omega$ be any representation given by Proposition \ref{omega}. Then there exists an admissible unitary $L$-Banach space representation $(E, \|\centerdot\|)$ of $G$, such that 
$\| E\|\subset |L|$, $\bigl (\begin{smallmatrix} p & 0 \\ 0 & p \end{smallmatrix}\bigr )$
acts trivially on $E$, and the reduction modulo $\varpi$ of the unit ball in $E$ is isomorphic to $\Omega$ as a $G$-representation.
\end{thm}
\begin{proof} If $p\neq 2$ this is shown in \cite[Thm.6.1]{comp}. We will observe that the renaming trick allows us to carry out 
essentially the same proof when $p=2$. We make no assumption on $p$. %Let $\Omega$ be any representation given by Proposition \ref{omega}.

We first lift $\Omega|_K$ to characteristic $0$. Let $\sigma$ be the $K$-socle of $\Omega$. Pontryagin duality induces an anti-equivalence of categories between $\Mod^{\sm}_K(k)$ and the category of pseudocompact $k\br{K}$-modules, which we denote by $\Mod^{\mathrm{pro.aug}}_K(k)$. Since $\Omega$ is an injective envelope of $\sigma$ in $\Mod^{\sm}_K(k)$, 
its Pontryagin dual $\Omega^{\vee}$ is a projective envelope of $\sigma^{\vee}$ in $\Mod^{\mathrm{pro.aug}}_K(k)$. Let $\wP_{\sigma^{\vee}}$ be a projective envelope of $\sigma^{\vee}$ in the category of pseudocompact $\OO\br{K}$-modules. Then $\wP_{\sigma^{\vee}}/\varpi \wP_{\sigma^{\vee}}$ is a  projective envelope of $\sigma^{\vee}$ in $\Mod^{\mathrm{pro.aug}}_K(k)$. Since projective envelopes are unique up to isomorphism, we obtain 
$\Omega^{\vee}\cong \wP_{\sigma^{\vee}}/\varpi \wP_{\sigma^{\vee}}$. Since $\tau_1$ is admissible and $\sigma\cong \soc_K \tau_1$ by Corollary \ref{socleomega}, $\sigma$ is a
finite dimensional $k$-vector space. In particular, $\sigma^{\vee}$ is a finitely generated $\OO\br {K}$-module, and so there exists a surjection 
of $\OO\br{K}$-modules $\OO\br{K}^{\oplus r}\twoheadrightarrow \sigma^{\vee}$. Since $\OO\br{K}^{\oplus r}$ is projective, and 
$\wP_{\sigma^{\vee}}\twoheadrightarrow \sigma^{\vee}$ is essential, the surjection factors through $\OO\br{K}^{\oplus r}\twoheadrightarrow \wP_{\sigma^{\vee}}$, and so $\wP_{\sigma^{\vee}}$ is a finitely generated $\OO\br{K}$-module. Since $\wP_{\sigma^{\vee}}$ is projective, 
we deduce that it is a direct summand of $\OO\br{K}^{\oplus r}$, and hence it is $\OO$-torsion free.

Thus $\wP_{\sigma^{\vee}}$ is an $\OO$-torsion free, finitely generated $\OO\br{K}$-module, and its reduction modulo $\varpi$ is isomorphic 
to $\Omega^{\vee}$ in $\Mod^{\mathrm{pro.aug}}_K(k)$. Let $E_0=\Hom_{\OO}^{cont}(\wP_{\sigma^{\vee}}, L)$, and let $\|\centerdot\|_0$ be the supremum norm.
It follows from \cite{iw} that $E_0$ is an admissible unitary $L$-Banach space representation of $K$. Moreover, the unit ball $E^0_0$ in $E_0$ 
is $\Hom_{\OO}^{cont}(\wP_{\sigma^{\vee}}, \OO)$ and 
$$ \Hom_{\OO}^{cont}(\wP_{\sigma^{\vee}}, \OO)\otimes_{\OO} k\cong  \Hom_{\OO}^{cont}(\wP_{\sigma^{\vee}}, k)\cong 
\Hom_k^{cont}(P_{\sigma^{\vee}}, k)\cong (\Omega^{\vee})^{\vee}\cong \Omega,$$ 
see \cite[\S5]{comp} for details. We extend the action of $K$ on $E_0$ to the action of $KZ$ by letting 
$\bigl (\begin{smallmatrix} p & 0 \\ 0 & p \end{smallmatrix}\bigr )$ act trivially.

Since $\sigma$ is finite dimensional, it follows from \cite[Lem.6.2.4]{coeff} that $\Omega^{I_1}$  is  a finite dimensional $k$-vector space. Since 
$\Omega|_{\KK}\cong (\Indu{\GG}{\KK}{\Omega^{I_1}})_{\sm}$ by  Proposition \ref{omega}, Lemma \ref{lift_it} implies 
that there exists a unitary $L$-Banach space representation $(E_1, \|\centerdot\|_1)$ of $\mathfrak K$, such that $\|E_1\|\subseteq |L|$, 
$\bigl (\begin{smallmatrix} p & 0 \\ 0 & p \end{smallmatrix}\bigr )$
acts trivially on $E_1$ and the reduction of the unit ball $E_1^0$ in $E_1$ modulo $\varpi$ is isomorphic to $\Omega|_{\mathfrak K}$. We claim that 
there exists an isometric, $IZ$-equivariant isomorphism $\varphi: E_1\rightarrow E_0$ such that the following diagram of $IZ$-representations:
\begin{equation}\label{beauty}
\xymatrix{ E^0_1/\varpi E^0_1\ar[r]^-{\varphi}_-{\mod{\varpi}}\ar[d]^-{\cong}& E^0_0/\varpi E^0_0\ar[d]^-{\cong}\\
 \Omega \ar[r]^-{\id} &\Omega}
\end{equation}
commutes,  where the left vertical arrow is  the given $\KK$-equivariant isomorphism $E^0_1/\varpi E^0_1\cong \Omega|_{\KK}$ and 
the  right vertical arrow is the given $KZ$-equivariant isomorphism $E^0_0/\varpi E^0_0\cong \Omega|_{KZ}$. Granting the claim, 
we may transport the action of  $\KK$ on $E_0$ by using $\varphi$ to obtain a unitary action of $KZ$ and $\KK$ on $E_0$, such that 
the two actions agree on $KZ\cap \KK$, which is equal to $IZ$. The resulting action glues to the unitary action of $G$ on $E_0$, see 
\cite[Cor.5.5.5]{coeff}, which is stated for smooth representations, but the proof of which works for any representation.  
%since $G$ is an amalgam of $KZ$ and $\KK$ along $IZ$.
 The commutativity of the above diagram implies that $E^0_0\otimes_{\OO} k\cong \Omega$ as a $G$-representation. 

We will prove the claim now. Let 
$M=\Hom^{cont}_{\OO}(E^0_1, \OO)$ equipped with the topology of pointwise convergence. Then $M$ is an object of $\Mod^{\mathrm{pro.aug}}_{I}(\OO)$,
and $M\otimes_{\OO} k \cong \Omega^{\vee}$ in  $\Mod^{\mathrm{pro.aug}}_{I}(k)$, see \cite[Lem.5.4]{comp}. Since $\Omega|_K$ is injective 
in $\Mod^{\sm}_K(k)$, $\Omega|_I$ is injective in $\Mod^{\sm}_I(k)$. Since $I_1$ is a pro-$p$ group, every non-zero $I$-invariant subspace
of $\Omega$ intersects $\Omega^{I_1}$ non-trivially. Thus $\Omega|_I$ is an injective envelope of $\Omega^{I_1}$ in  $\Mod^{\sm}_I(k)$. 
Hence, $\Omega^{\vee}$ is a projective envelope of $(\Omega^{I_1})^{\vee}$ in $\Mod^{\mathrm{pro.aug}}_I(k)$. Since $M$ is $\OO$-torsion free, 
and $M\otimes_{\OO} k$ is a projective envelope of $(\Omega^{I_1})^\vee$ in $\Mod^{\mathrm{pro.aug}}_I(k)$, \cite[Prop.4.6]{comp} implies that 
$M$ is a projective envelope of $(\Omega^{I_1})^\vee$ in $\Mod^{\mathrm{pro.aug}}_I(\OO)$. The same holds for $\wP_{\sigma^{\vee}}$. Since projective envelopes
are unique up to isomorphism, there exists an isomorphism $\psi: \wP_{\sigma^{\vee}}\overset{\cong}{\rightarrow} M$ in $\Mod^{\mathrm{pro.aug}}_I(\OO)$. 
It follows from \cite[Cor.4.7]{comp} that the natural map $\Aut_{\OO\br{I}}(\wP_{\sigma^{\vee}})\rightarrow 
\Aut_{k\br{I}}(\wP_{\sigma^{\vee}}/\varpi \wP_{\sigma^{\vee}})$ is surjective. Using this we may choose $\psi$ so that the following diagram in 
$\Mod^{\mathrm{pro.aug}}_{K}(k)$: 
 \begin{equation}\label{beauty2}
\xymatrix{ \wP_{\sigma^{\vee}}/\varpi \wP_{\sigma^{\vee}}\ar[r]^-{\psi}_-{\mod{\varpi}}\ar[d]^-{\cong}& M/\varpi M\ar[d]^-{\cong}\\
 \Omega^{\vee} \ar[r]^-{\id} &\Omega^{\vee}}
\end{equation}
commutes. Dually we obtain an isometric $I$-equivariant isomorphism of unitary $L$-Banach space representations of $I$, 
$\psi^d:  \Hom_{\OO}^{cont}(M, L)\rightarrow \Hom_{\OO}^{cont}(\wP_{\sigma^{\vee}}, L)$. It follows from \cite[Thm.1.2]{iw} that 
$(E_1, \|\centerdot\|_1)$ is naturally and isometrically isomorphic to $\Hom^{cont}_{\OO}(M, L)$ with the supremum norm.
This gives our $\varphi$.  The commutativity of \eqref{beauty2} implies the commutativity of \eqref{beauty}.
\end{proof}

\begin{cor}\label{summand} The Banach space representation $(E, \|\centerdot\|)$ constructed in Theorem \ref{lift} is isometrically, $K$-equivariantly isomorphic to 
a direct summand of $\mathcal C(K, L)^{\oplus r}$, where $\mathcal C(K, L)$ is the space of continuous functions 
from $K$ to $L$, equipped with the supremum norm,  and $r$ is a positive integer.
\end{cor}
\begin{proof} It follows from the construction of $E$, that $(E,\|\centerdot\|)$ is isometrically, $K$-equivariantly isomorphic to 
$\Hom_{\OO}^{cont}(\wP_{\sigma^{\vee}}, L)$ with the supremum norm. Moreover, it follows from the proof of Theorem \ref{lift} that  
$\wP_{\sigma^{\vee}}$ is a direct summand of $\OO\br{K}^{\oplus r}$. It is shown in \cite[Lem.2.1, Cor.2.2]{iw} 
that the natural map $K\rightarrow \OO\br{K}$, $g\mapsto g$ induces an isometrical, $K$-equivariant isomorphism between
$\mathcal C(K, L)$ and $\Hom_{\OO}^{cont}(\OO\br{K}, L)$.
\end{proof}
 
If $F$ is a finite extension of $\Qp$  then  exactly the same proof works. We note that \cite[Thm.9.8]{bp} is proved for $\GL_2(F)$.
We record this as a corollary below. Let 
$\OO_F$ be the ring of integers of $F$, $\varpi_F$ a uniformizer, $k_F$ the residue field, let $\GG_F$ be the subgroup of 
$\GL_2(F)$ generated by the matrices $\bigl (\begin{smallmatrix} \varpi_F & 0 \\ 0 & \varpi_F \end{smallmatrix}\bigr )$, 
$\bigl (\begin{smallmatrix} 0 & 1  \\  \varpi_F & 0\end{smallmatrix}\bigr )$ and 
 $\bigl (\begin{smallmatrix} [\lambda] & 0 \\ 0 & [\mu] \end{smallmatrix}\bigr )$, for $\lambda, \mu \in k_F^{\times}$, where 
$[\lambda]$ is the Teichm\"uller lift of $\lambda$. Let $I_1$ be the standard pro-$p$ Iwahori subgroup of $G$.

\begin{cor} Let $\tau$ be an admissible smooth $k$-representation of $\GL_2(F)$, such that 
$\bigl (\begin{smallmatrix} \varpi_F & 0 \\ 0 & \varpi_F \end{smallmatrix}\bigr )$ acts trivially on $\tau$ and if $p=2$ assume that
the inclusion $\tau^{I_1}\hookrightarrow \tau$ has a $\GG_F$-equivariant section. Then there exists a $\GL_2(F)$-equivariant embedding 
$\tau\hookrightarrow \Omega$, such that  $\Omega|_{\GL_2(\OO_F)}$ is an injective envelope of $\GL_2(\OO_F)$-socle of $\tau$ in the category of smooth $k$-representations of $\GL_2(\OO_F)$  and $\bigl (\begin{smallmatrix} \varpi_F & 0 \\ 0 & \varpi_F \end{smallmatrix}\bigr )$ acts trivially 
on $\Omega$. Moreover, we may lift $\Omega$ to an admissible unitary $L$-Banach space representation of $\GL_2(F)$.
\end{cor}

\begin{remar} We also note that one could work with a fixed central character throughout.
\end{remar}

Let $V_{l,k}=\det^l \otimes \Sym^{k-1}L^2$, for $k\in \NN$ and $l\in \ZZ$. Rather unfortunately $k$ 
also denotes the residue field of $L$, we hope that this will not cause any confusion.

\begin{prop}\label{Kalg} Let $(E, \|\centerdot\|)$ be a unitary $L$-Banach space representation of $K$ isomorphic in the category of unitary admissible
$L$-Banach space representations of $K$ to a direct summand of $\mathcal C(K, L)^{\oplus r}$.  The evaluation  map
\begin{equation}\label{bigsum}
 \bigoplus_{(l,k)\in \ZZ\times \NN}  \Hom_K(V_{l,k}, E)\otimes V_{l,k}  \rightarrow  E
\end{equation}
is injective and the image is a dense subspace of   $E$. % In particular, the image is a dense subspace of $E$.
Moreover, the subspaces  $\Hom_K(V_{l,k}, E)$ are finite dimensional. 
\end{prop}
\begin{proof} The argument is  the same as given in the proof of \cite[Prop.5.4.1]{emfm}.  We have provided 
some details in the Appendix at the request of the referee.  It is enough to prove the statement for $\mathcal C(K, L)$, since then it is true for $\mathcal C(K, L)^{\oplus r}$ and by applying the 
idempotent, which cuts out $E$, we may deduce the same statement for $E$. In the case $E=\mathcal C(K, L)$, the assertion 
follows from Proposition \ref{A3} applied to $G=\GL_2$. We note that every rational irreducible representation of $\GL_{2}/L$ is isomorphic to $V_{l,k}$ for a unique pair $(l, k)\in \ZZ\times \NN$. The last assertion follows from \eqref{map3} below.
\end{proof}

\begin{prop}\label{universal} Let $\rho=(\Indu{B}{G}{\chi_1\otimes\chi_2|\centerdot|^{-1}})_{\sm}$ be a smooth  principal series representation of $G$, 
where  $\chi_1,\chi_2:\Qp^{\times}\rightarrow L^{\times}$ smooth characters with $\chi_1|\centerdot|\neq \chi_2$. 
Let $\Pi$ be the universal unitary completion of $\rho\otimes V_{l,k}$. Then $\Pi$ is an admissible, finite length $L$-Banach space  representation of $G$.  Moreover, 
if $\Pi$ is non-zero and we let $\overline{\Pi}^{ss}$ be the semi-simplification of the reduction modulo $\varpi$ of 
an open bounded $G$-invariant lattice in $\Pi$,  then either $\overline{\Pi}^{ss}$ is irreducible supersingular, or 
\begin{equation}\label{delta12}
\overline{\Pi}^{ss} \subseteq (\Indu{B}{G}{\delta_1\otimes\delta_2\omega^{-1}})^{ss}_{\sm} \oplus (\Indu{B}{G}{\delta_2\otimes\delta_1\omega^{-1}})^{ss}_{\sm},
\end{equation} 
for some smooth characters $\delta_1, \delta_2: \Qp^{\times}\rightarrow k^{\times}$, where the superscript $ss$ indicates the semi-simplification.
\end{prop}
\begin{proof} If $\Pi\neq 0$ then $-(k+l)\le \val(\chi_1(p))\le -l$, $-(k+l)\le \val(\chi_2(p))\le -l$ and 
$\val(\chi_1(p))+\val(\chi_2(p))=-(k+2l)$,
\cite[Lem.7.9]{comp}, \cite[Lem.2.1]{pem}.  If both inequalities are strict and $\chi_1\neq \chi_2$ then it is shown in
\cite[\S5.3]{bb} that $\Pi$ is non-zero, admissible and absolutely irreducible. The assertion about $\overline{\Pi}^{ss}$ then 
follows from 
\cite{berger}. 

If both inequalities are strict, $\chi_1= \chi_2$ and $\Pi$ is non-zero it is shown in \cite[Prop.4.2]{except} that 
there exist  $\OO$-lattices $M$ in $\rho\otimes V_{l,k}$ and $M'$ in $\rho'\otimes V_{l,k}$, where 
$\rho'=(\Indu{B}{G}{\chi_1'\otimes\chi_2'|\centerdot|^{-1}})_{\sm}$ for some distinct smooth characters, 
$\chi_1', \chi_2':\Qp^{\times}\rightarrow L^{\times}$ congruent to $\chi_1, \chi_2$ modulo $1+(\varpi)$, such that 
both lattices are  finitely generated $\OO[G]$-modules and their reductions modulo $\varpi$ are isomorphic. Since $M$ is $\OO$-torsion 
free, the completion of $\rho\otimes V_{l,k}$ with respect to the gauge of $M$ is non-zero, and since $M$ is a finitely generated 
$\OO[G]$-module, the completion is the universal unitary completion, \cite[Prop.1.17]{pem}, 
thus is isomorphic to $\Pi$. Let $\Pi^0$ be the unit ball in $\Pi$
with respect to the gauge of $M$. Then $\Pi^0/\varpi \Pi^0\cong M/\varpi M\cong M'/\varpi M'$. Now by the same argument 
the completion of $\rho'\otimes V_{l,k}$ with respect to the gauge of $M'$ is the universal unitary completion of 
of $\rho'\otimes V_{l,k}$. Since $\chi_1'\neq \chi_2'$ we may apply the results of Berger-Breuil \cite{bb} to conclude that 
the semi-simplification of $M'/\varpi M'$ has the desired form.

Suppose that either $\val(\chi_1(p))=-l$ or $\val(\chi_2(p))=-l$. If $\chi_1=\chi_2|\centerdot|$ then this forces $k=1$, so that
$V_{l,k}$ is a character and $\rho\otimes V_{l, k}\cong (\Indu{B}{G}{|\centerdot|\otimes|\centerdot|^{-1}})_{\sm}\otimes \eta$, where 
$\eta:G\rightarrow L^{\times}$ is a unitary character. It follows from \cite[Lem.5.3.18]{emcoates} that the universal unitary completion
of $\rho\otimes V_{l,k}$ is admissible and of length $2$. Moreover, 
$\overline{\Pi}^{ss}\cong \overline{\eta}\oplus \Sp \otimes \overline{\eta}\cong 
(\Indu{B}{G}{\overline{\eta}\otimes\overline{\eta}})_{\sm}^{ss}$. If $\chi_1\neq \chi_2|\centerdot|$ then it follows from 
\cite[Lem.2.2.1]{be} that the universal unitary completion of $\rho\otimes V_{l, k}$ is isomorphic 
to a continuous induction of a unitary character. Hence $\overline{\Pi}^{ss}$ is isomorphic to the semi-simplification of a principal series
representation. 
\end{proof}

\begin{proof}[Proof of Theorem \ref{main}] Let $(E, \|\centerdot\|)$ be the unitary $L$-Banach space representation of $G$ constructed 
in the proof of Theorem \ref{lift}. Let $E^0$ be the unit ball in $E$, then by construction we have $E^0/\varpi E^0\cong \Omega$, where
$\Omega$ is a smooth $k$-representation of $G$, satisfying the conditions of Proposition \ref{omega}.  Let 
$V=\oplus \Hom_K(V_{l,k}, E)\otimes V_{l, k}$, where the sum is taken over all $(l, k)\in \ZZ\times \NN$. It follows from 
Corollary \ref{summand} and Proposition \ref{Kalg} that the natural map $V\rightarrow E$ is injective and the image is dense. 
Let $\{V^i\}_{i\ge 0}$ be any increasing, exhaustive filtration of $V$ by finite dimensional $K$-invariant subspaces. Then $V^i\cap E^0$ 
is a $K$-invariant $\OO$-lattice in $V^i$, and we denote by $\overline{V}^i$ its reduction modulo $\varpi$. It follows from \cite[Lem.5.5]{comp} 
that the reduction modulo $\varpi$ induces a $K$-equivariant  injection $\overline{V}^i\hookrightarrow \Omega$. The density of $V$ in $E$ 
implies that  $\{\overline{V}^i\}_{i\ge 0}$ is an increasing, exhaustive filtration of $\Omega$ by finite dimensional, $K$-invariant 
subspaces. Recall that $\Omega$ contains $\tau$ as a subrepresentation, see Proposition \ref{omega}. Now  $\tau$ is finitely generated 
as a $G$-representation, since it is of finite length. Thus we may conclude, that  there exists a finite dimensional 
$K$-invariant subspace $W$ of $V$, such that $\tau$ is contained in the $G$-subrepresentation of $\Omega$ generated by $\overline{W}$.

Let $\varphi: V_{l,k}\rightarrow E$ be a non-zero $K$-equivariant, $L$-linear homomorphism. Let $R(\varphi)$ be the  $G$-sub\-rep\-re\-sen\-ta\-tion  
of $E$ in the category of (abstract) $G$-rep\-re\-sen\-ta\-tions on $L$-vector spaces, 
generated by the image of $\varphi$. Frobenius reciprocity gives us 
a surjection $\cIndu{KZ}{G}{\tilde{\Eins}}\otimes V_{l,k}\twoheadrightarrow R(\varphi)$, where $\tilde{\Eins}: KZ\rightarrow L^{\times}$ is 
an unramified character, such that $\bigl (\begin{smallmatrix} p & 0 \\ 0 & p\end{smallmatrix} \bigr)$ acts trivially on 
$V_{l,k} \otimes \tilde{\Eins}$. Now $\End_G(\cIndu{KZ}{G}{\tilde{\Eins}})$ is isomorphic to the ring of polynomials over $L$ 
in one variable $T$. It follows from the proof of \cite[Cor.7.4]{comp} that the surjection factors through 
$ \frac{\cIndu{KZ}{G}{\tilde{\Eins}}}{P(T)}\otimes V_{l,k}\twoheadrightarrow R(\varphi)$, for some non-zero $P(T)\in L[T]$.

Let $R$ be the (abstract) $G$-subrepresentation of $E$ generated by $W$, and let $\Pi$ be the closure of $R$ in $E$. Since $W$
is isomorphic to a finite direct sum of $V_{l,k}$'s, we deduce that if we replace $L$ by a finite extension there exists a 
surjection: 
\begin{equation}\label{surjR} 
\bigoplus_{i=1}^m \frac{\cIndu{KZ}{G}{\tilde{\Eins}_i}}{(T-a_i)^{n_i}}\otimes V_{l_i, k_i} \twoheadrightarrow R,
\end{equation}
for some $a_i\in L$, $n_i\in \NN$ and $(l_i,k_i)\in \ZZ\times \NN$. Let $\rho_i= \frac{\cIndu{KZ}{G}{\tilde{\Eins}_i}}{T-a_i}$, then
using \eqref{surjR} we may construct a finite, increasing, exhaustive  filtration $\{R^j\}_{j\ge 0}$ of $R$ by $G$-invariant subspaces, 
such that for each $j$ there exists a surjection $\rho_i\otimes V_{l_i, k_i}\twoheadrightarrow R^{j}/R^{j-1}$, for some $1\le i\le m$. 
Moreover, by choosing $n_i$ and $m$ in \eqref{surjR} to be minimal, we may assume that $\Hom_G(\rho_i\otimes V_{l_i, k_i}, R)$ 
is non-zero for all $1\le i\le m$. Let $\Pi^j$ be the closure of $R^j$ in $E$. We note that since $E$ is admissible, $\Pi^j$ is an
admissible unitary $L$-Banach space representation of $G$, moreover the category $\Ban^{\adm}_G(L)$ is abelian. Since $R^j$ is dense in 
$\Pi^j$, its image is dense in $\Pi^j/\Pi^{j-1}$. Hence, for each $j$  there exists  a $G$-equivariant map 
$\varphi_j:\rho_i\otimes V_{l_i, k_i}\rightarrow \Pi^j/\Pi^{j-1}$ with a dense image. Let $\Pi_i$ be the universal unitary completion 
of $\rho_i\otimes V_{l_i, k_i}$. Since the target of $\varphi_j$ is unitary, we can extend it to a continuous $G$-equivariant map 
$\tilde{\varphi}_j: \Pi_i\rightarrow \Pi^j/\Pi^{j-1}$. Moreover, since the target of $\varphi_j$ is admissible and 
the image is dense, $\tilde{\varphi}_j$ is surjective. 

For each closed subspace $U$ of $E$, we let $\overline{U}$ be the reduction of $(U\cap E^0)$ modulo $\varpi$. It follows from 
\cite[Lem.5.5]{comp} that the reduction modulo $\varpi$ induces an injection $\overline{U}\hookrightarrow \Omega$.
Since $\Pi$ contains $W$, $\overline{\Pi}$ will contain $\overline{W}$. Since $\overline{\Pi}$ is $G$-invariant, it will contain 
$\tau$. Now $\{\overline{\Pi}^j\}_{j\ge 0}$ defines a finite,  increasing, exhaustive filtration of $\overline{\Pi}$ by $G$-invariant subspaces. 
Since $\pi_2$ is an irreducible  subquotient of $\tau$, there exists $j$, such that $\pi_2$ is an irreducible subquotient of 
$\overline{\Pi}^j/\overline{\Pi}^{j-1}$. 

Each representation $\rho_i$ is an unramified principal series representation, considered in Proposition \ref{universal}, see 
\cite[Prop.3.2.1]{breuil2}. Hence, $\Pi_i$ is an admissible, finite length $L$-Banach space representation of $G$, moreover 
$\overline{\Pi}^{ss}_i$ is of finite length as described in Proposition \ref{universal}. The surjection 
$\tilde{\varphi}_j: \Pi_i\twoheadrightarrow \Pi^j/\Pi^{j-1}$ induces a surjection 
$\overline{\Pi}^{ss}_i\twoheadrightarrow \overline{(\Pi^j/\Pi^{j-1})}^{ss}$. It follows from \cite[Lem.5.5]{comp} that 
the semi-simplification of $\overline{\Pi}^j/\overline{\Pi}^{j-1}$ is isomorphic to $\overline{(\Pi^j/\Pi^{j-1})}^{ss}$.
Thus $\pi_2$ is a subquotient of $\overline{\Pi}^{ss}_i$. 

Since $\Hom_G(\rho_i\otimes V_{l_i, k_i}, \Pi)$ is non-zero, there exists a non-zero continuous $G$-in\-va\-riant homomorphism 
$\varphi:\Pi_i\rightarrow  \Pi$. Let $\Sigma$ be the image of $\varphi$. Since $\Pi_i$ and $\Pi$ are admissible, we have 
a surjection $\Pi_i\twoheadrightarrow \Sigma$ and an injection $\Sigma\hookrightarrow \Pi$ in the abelian category $\Ban^{\adm}_G(L)$.
The surjection induces a surjection $\overline{\Pi}_i^{ss}\twoheadrightarrow \overline{\Sigma}^{ss}$. The injection induces an 
injection $\overline{\Sigma}\hookrightarrow \overline{\Pi}\hookrightarrow \Omega$. Since $\soc_G \Omega\cong \pi_1$ by Corollary
\ref{socleomega} and $\overline{\Sigma}$ is non-zero, we deduce that $\pi_1\cong \soc_G \overline{\Sigma}$. Hence, $\pi_1$ is
 a subquotient of $\overline{\Pi}^{ss}_i$.  
\end{proof}

\begin{lem}\label{scalars} 
Let $\kappa$ and $\lambda$ be smooth $k$-representations of $G$ and let $l$ be a finite extension of $k$. Then 
$\Ext^i_G(\kappa, \lambda)\otimes_k l \cong \Ext^i_G(\kappa\otimes_k l, \lambda\otimes_k l)$, for all $i\ge 0$, where 
the $\Ext$ groups are computed in $\Mod^{\sm}_G(k)$ and $\Mod^{\sm}_G(l)$, respectively.
\end{lem}
\begin{proof} The assertion for $i=0$ follows from \cite[Lem.5.1]{cmf}. Hence, it is enough to find an injective resolution of 
$\lambda$ in $\Mod^{\sm}_G(k)$, which remains injective after tensoring with $l$. Such resolution may be obtained by considering
$(\Indu{\{1\}}{G}{V})_{\sm}$, where $\{1\}$ is the trivial subgroup of $G$ and $V$ is a $k$-vector space. We note that 
$(\Indu{\{1\}}{G}{V})_{\sm}\otimes_k l\cong (\Indu{\{1\}}{G}{V\otimes_k l})_{\sm}$, since $l$ is finite over $k$. 
\end{proof}

\begin{proof}[Proof of Corollary \ref{blocks}] Lemma \ref{scalars} implies that replacing $L$ by a finite extension does not 
change the blocks. It follows from Proposition \ref{universal} and Theorem \ref{main} that
an irreducible  supersingular representation is in a block on its own. Let $\pi\{\delta_1, \delta_2\}$ be  the semi-simple representation 
defined by \eqref{delta12}, where $\delta_1, \delta_2:\Qp^{\times}\rightarrow k^{\times}$ are smooth characters. We have to show that all 
irreducible subquotients of  $\pi\{\delta_1, \delta_2\}$ lie in the same block. We adopt an argument used in \cite{colmez}. 
It follows from \cite[5.3.3.1, 5.3.3.2, 5.3.4.1]{breuil2} that there exists an irreducible unitary $L$-Banach space representation 
$\Pi$ of $G$, such that $\overline{\Pi}^{ss}\cong \pi\{\delta_1, \delta_2\}$, then \cite[Prop.VII.4.5(i)]{colmez} asserts that
we may choose an open bounded $G$-invariant lattice $\Theta$ in $\Pi$ such that $\Theta/\varpi \Theta$ is indecomposable. It follows from 
\eqref{blocksdecompose} that all the irreducible subquotients of $\Theta/\varpi \Theta$ lie in the same block. 

We will list explicitly the irreducible subquotients of $\pi\{\delta_1, \delta_2\}$. It is shown in \cite{bl} 
that if $\delta_2 \delta_1^{-1}\neq \omega$ then $(\Indu{B}{G}{\delta_1\otimes\delta_2\omega^{-1}})_{\sm}$ is absolutely irreducible, 
and there exists a non-split exact sequence 
\begin{equation}\label{spblock}
0\rightarrow \delta_1\circ \det \rightarrow  (\Indu{B}{G}{\delta_1\otimes\delta_2\omega^{-1}})_{\sm}\rightarrow \Sp\otimes\delta_1\circ \det
\rightarrow 0
\end{equation} 
if $\delta_2\delta_1^{-1}=\omega$. Taking this into account there are the following possibilities for decomposing 
$\pi\{\delta_1, \delta_2\}$ into irreducible direct summands depending 
on $\delta_1$, $\delta_2$ and $p$:
\begin{itemize} 
\item[(i)] If $\delta_2\delta_1^{-1}\neq \omega^{\pm 1}, \Eins$ then
 $$\pi\{\delta_1, \delta_2\}\cong (\Indu{B}{G}{\delta_1\otimes\delta_2\omega^{-1}})_{\sm} \oplus 
(\Indu{B}{G}{\delta_2\otimes\delta_1\omega^{-1}})_{\sm};$$
\item[(ii)] if $\delta_2=\delta_1(=\delta)$ then 
\begin{itemize}
\item[(a)] if $p>2$ then $\pi\{\delta, \delta\}\cong (\Indu{B}{G}{\delta\otimes\delta\omega^{-1}})_{\sm}^{\oplus 2}$;
\item[(b)] if $p=2$ then $\pi\{\delta, \delta\} \cong (\Sp^{\oplus 2} \oplus \Eins^{\oplus 2})\otimes \delta\circ \det$.
\end{itemize}
\item[(iii)] if $\delta_2\delta_1^{-1}=\omega^{\pm 1}$ then 
\begin{itemize} 
\item[(a)] if $p\ge 5$ then 
$\pi\{\delta_1, \delta_2\}\cong (\Eins\oplus \Sp\oplus (\Indu{B}{G}{\omega\otimes \omega^{-1}})_{\sm})\otimes \delta\circ \det$;
\item[(b)] if $p=3$ then 
$\pi\{\delta_1, \delta_2\}\cong (\Eins\oplus \Sp\oplus \omega\circ \det \oplus \Sp\otimes \omega\circ \det)\otimes \delta\circ \det;$
\item[(c)] if $p=2$ then we are in the case (ii)(b), 
\end{itemize}
where $\delta$ is either $\delta_1$ or $\delta_2$.
\end{itemize}
Finally, we note that in the case (ii)(b) instead of using \cite[5.3.3.2]{breuil2}, which is stated without proof, we could have observed 
that since \eqref{spblock} is non-split,  $\Sp\otimes\delta_1\circ \det$ and $\delta_1\circ\det$ lie in the same block. 
\end{proof}

\appendix
\section{Density of algebraic vectors}

Let $X$ be an affine  scheme of finite type over $\Zp$ and let $A=\Gamma(X, \mathcal O_X)$. By choosing an isomorphism $A\cong\Zp[x_1, \ldots, x_n]/(f_1, \ldots, f_m)$ we may identify 
the $X(\Zp)$ with a closed subset of $\Zp^n$. The induced topology on $X(\Zp)$ is independent of a choice of the isomorphism, see \cite[Prop.2.1]{conrad}. Let  $\mathcal C(X(\Zp), L)$ be the space of continuous functions from $X(\Zp)$ to $L$.
 Since $X(\Zp)$ is  compact, $\mathcal C(X(\Zp), L)$ equipped with the supremum norm is an $L$-Banach space. Recall that $X(\Zp)= \Hom_{\Zp-alg}(A, \Zp)$.  We denote by 
 $\mathcal C^{\alg}(X(\Zp), L)$ the functions $f: X(\Zp)\rightarrow L$, which are obtained by evaluating elements of $A\otimes_{\Zp} L$ at  $\Zp$-valued points of $X$. 

\begin{lem}\label{dense_algebraic} $\mathcal C^{\alg}(X(\Zp), L)$ is a dense subspace of $\mathcal C(X(\Zp), L)$.
\end{lem}
\begin{proof} We first look at the special case, when $X=\mathbb A^n$, so that $A=\Zp[x_1, \ldots, x_n]$ and $X(\Zp)=\Zp^{n}$. Since addition and multiplication 
in $\Zp$ are continuous functions, we deduce that  $\mathcal C^{\alg}(\mathbb A(\Zp), L)$ is a subspace of $\mathcal C(\mathbb A(\Zp), L)$. The density follows from the theory of 
Mahler expansions, see for example \cite[III.1.2.4]{laz}. In the general case, we choose an isomorphism $A\cong\Zp[x_1, \ldots, x_n]/(f_1, \ldots, f_m)$ and  identify $X(\Zp)$ with a closed subset of $\mathbb A^n(\Zp)=\Zp^{ n}$. %The natural topology on $X(\Zp)$ coincides with the subspace topology, see \cite[Prop.2.1]{conrad}.
The restriction of functions to $X(\Zp)$ induces a surjective map $r: \mathcal C(\mathbb A^n(\Zp), L)\rightarrow \mathcal C(X(\Zp), L)$, see for example \cite[Thm.3.1(1)]{ellis}.
Since $\mathcal C^{\alg}(\mathbb A^n(\Zp), L)$ is dense in $\mathcal C(\mathbb A^n(\Zp), L)$ and $\sup_{x\in X(\Zp)} |r(f)(x)|\le \sup_{x\in \mathbb A^n(\Zp)}| f(x)|$ for all $f\in \mathcal C(\mathbb A^n(\Zp), L)$
we deduce  that $r( \mathcal C^{\alg}(\mathbb A^n(\Zp), L))$ is a dense subspace. Since it is equal to $\mathcal C^{\alg}(X(\Zp), L)$ we are done.
\end{proof}

\begin{remar} If $X$ is an affine scheme of finite type over $\mathcal O_F$,  where  $\mathcal O_F$ is a ring of integers in a finite field extension $F$ over $\Qp$, then 
there are two ways to topologize $X(\mathcal O_F)$: as $\mathcal O_F$ points of $X$ and as $\Zp$-points of the Weil restriction of $X$ to $\Zp$. However, they  coincide, see \cite[Ex.2.4]{conrad}.
\end{remar}

\begin{prop}\label{A3} Let $G$ be an affine group scheme of finite type over $\Zp$ such that $G_L$ is a split connected reductive group over $L$. Then  the evaluation map
\begin{equation}\label{map}
\bigoplus_{[V]} \Hom_{G(\Zp)}(V, \mathcal C(G(\Zp), L))\otimes V \rightarrow \mathcal C(G(\Zp), L),
\end{equation}
where the sum is taken over all the isomorphism classes of irreducible rational representations of $G_L$, is injective and the image is equal to $\mathcal C^{\alg}(G(\Zp), L)$. In particular, 
the image of \eqref{map} is a dense subspace of  $\mathcal C(G(\Zp), L)$.
\end{prop}
\begin{proof} The category of rational representations of $G_L$ is semi-simple as $L$ is of characteristic $0$, see \cite[II.5.6 (6)]{jantzen}. Hence, 
the regular representation $\OO(G_L)$ decomposes into a direct sum of irreducible representations.  Since we have assumed that $G_L$ is split, every irreducible rational representation $V$ of $G_L$ is absolutely irreducible \cite[II.2.9]{jantzen}. This implies that $\End_{G_L}(V)=L$ for every irreducible representation $V$. 
It follows from Frobenius reciprocity \cite[I.3.7 (3)]{jantzen} and the semi-simplicity of $\OO(G_L)$ that we have an isomorphism of $G_L$-representations:
\begin{equation}\label{map1}
 \OO(G_L)\cong \bigoplus_{[V]} V^*\otimes V
 \end{equation}
where the  $G_L$-action on $V^*$ is trivial. The isomorphism 
\eqref{map1} is $G(L)$-equivariant, and hence $G(\Zp)$-equivariant, which gives us an isomorphism of $G(\Zp)$-rep\-re\-sen\-tations: 
\begin{equation}\label{map2}
\mathcal C^{\alg}(G(\Zp), L)\cong \bigoplus_{[V]} V^* \otimes V.
\end{equation}
The map $\varphi \mapsto [ v\mapsto \varphi(v)(1)]$ induces an isomorphism  
\begin{equation}\label{map3}
\Hom_{ G(\Zp)}(V, \mathcal C(G(\Zp), L))\cong V^*,
\end{equation} 
 with the inverse map given by $\ell \mapsto [v \mapsto [g\mapsto \ell( g v)]]$. Since every $V$ is a finite dimensional $L$-vector space, we conclude from 
 \eqref{map2}, \eqref{map3}
 that the injection
 \begin{equation}\label{map4}
\Hom_{G(\Zp)}(V, \mathcal C^{\alg}(G(\Zp), L))\hookrightarrow \Hom_{G(\Zp)}(V, \mathcal C(G(\Zp), L))
\end{equation}
is an isomorphism. Moreover, as a byproduct we obtain that $\End_{G(\Zp)}(V)=L$ and $\Hom_{G(\Zp)}(V, W)=0$, if $V$ and $W$ are non-isomorphic 
irreducible representations of $G_L$. We conclude that the evaluation map is injective, and the image is equal to $\mathcal C^{\alg}(G(\Zp), L)$. 
Lemma \ref{dense_algebraic} implies the last assertion.
\end{proof}

\end{document}